\documentclass[11pt]{article}
\usepackage{amsmath, amsfonts, amssymb, amsthm, mathtools, tikz, bm} 
\usepackage[size=footnotesize]{caption} 
\usepackage{xcolor} 
\usepackage[margin=1in]{geometry} 
\usepackage{bbm}
\usepackage{float} 
\usepackage{hyperref, fullpage} 
\usepackage{verbatim} 
\usetikzlibrary{positioning}

\newcommand{\card}[1]{\lvert #1\rvert} 
\newcommand{\vph}{\varphi}

\newcommand{\ZZ}{{\mathbb Z}}
\newcommand{\NN}{{\mathbb N}}
\newcommand{\mad}{{\textrm{mad}}}

\def\C{\mathcal{C}}
\newtheorem{theorem}{Theorem}[section]
	\newtheorem{lemma}[theorem]{Lemma}
	
	\newtheorem{conj}[theorem]{Conjecture}
	
	\newtheorem{defn}{Definition}
	\theoremstyle{definition}
\usepackage{pgfplots}
\usetikzlibrary{intersections, pgfplots.fillbetween}
\tikzset{
    position/.style args={#1:#2 from #3}{
        at=(#3.#1), anchor=#1+180, shift=(#1:#2)
    }
}
\tikzset{Simple Node/.style={fill=black, circle, inner sep=2pt}}
\pgfplotsset{compat=1.17}
\pgfdeclarelayer{pre main}
\pgfdeclarelayer{foreground}
\pgfsetlayers{pre main, main,foreground}

\tikzstyle{uStyle}=[draw=black, fill=black, circle, thick, inner sep=1.2pt]

\def\aftermath{\par\vspace{-\belowdisplayskip}\vspace{-\parskip}\vspace{-\baselineskip}}

\title{The \texorpdfstring{{$t$}}{t}-Tone Chromatic Number\\ of
Classes of Sparse Graphs}
\author{Daniel W. Cranston \and Hudson LaFayette}
\author{Daniel W. Cranston\thanks{Department of Computer Science, Virginia Commonwealth University, Richmond, VA, USA; \texttt{dcranston@vcu.edu}} \and Hudson LaFayette\thanks{Department of Mathematics and Applied Mathematics, Virginia Commonwealth University, Richmond, VA, USA; \texttt{lafayettehl@vcu.edu}}}

\date{}

\begin{document}
\maketitle 

\begin{abstract}
For a graph \(G\) and \(t,k\in\ZZ^+\) a \emph{\(t\)-tone \(k\)-coloring} of
\(G\) is a function \(f:V(G)\rightarrow \binom{[k]}{t}\) such that
\(\card{f(v)\cap f(w)}<d(v,w)\) for all distinct \(v,w\in V(G)\).  
The \emph{\(t\)-tone chromatic number} of \(G\), denoted
\(\tau_t(G)\), is the minimum \(k\) such that \(G\) is \(t\)-tone \(k\)-colorable. 
For small values of $t$, we prove sharp or nearly sharp upper bounds on the
$t$-tone chromatic number of various classes of sparse graphs.
In particular, we determine $\tau_2(G)$ exactly when $\mad(G)<12/5$ and bound
$\tau_2(G)$, up to a small additive constant, when $G$ is outerplanar.
We also determine $\tau_t(C_n)$ exactly when $t\in\{3,4,5\}$.
\end{abstract}

\section{Introduction}
All of our graphs are finite and simple. We write \([k]\) to denote
\(\{1,\dots,k\}\) and write \(\binom{[k]}{t}\) to denote the collection of all 
subsets of \([k]\) of size \(t\); we refer to elements of \(\binom{[k]}{t}\) as
\emph{\(t\)-sets}. 
For a graph \(G\) and \(v,w\in V(G)\), we write \(d(v,w)\) for the distance (length of
the shortest path) between \(v\) and \(w\).

In 2009, Ping Zhang led N. Fonger, J. Goss, B.  Phillips, and C.
Segroves~\cite{fgps} in developing a new generalization of proper vertex
coloring. They called it \(t\)-tone coloring.

\begin{defn}
For a graph \(G\) and \(t,k\in\ZZ^+\) a \emph{\(t\)-tone \(k\)-coloring} of
\(G\) is a function \(f:V(G)\rightarrow \binom{[k]}{t}\) such that
\(\card{f(v)\cap f(w)}<d(v,w)\) for all distinct \(v,w\in V(G)\).  
A graph that has a \(t\)-tone \(k\)-coloring is \emph{\(t\)-tone
\(k\)-colorable}, and the \emph{\(t\)-tone chromatic number} of \(G\), denoted
\(\tau_t(G)\), is the minimum \(k\) such that \(G\) is \(t\)-tone \(k\)-colorable. 
\end{defn}

The most widely studied case of \(t\)-tone coloring is the case \(t=2\).
Fonger et al.~\cite{fgps} calculated the \(2\)-tone chromatic number for all
trees. This includes stars, which often provide a good lower 
bound for \(\tau_2(G)\); see Proposition~\ref{prop: 2-tone star lower bound}.
Bickle and Phillps~\cite{bp} determined, among other results, the \(2\)-tone
chromatic number of cycles and the general \(t\)-tone chromatic number of
paths; see Proposition~\ref{prop: t-tone coloring paths}. This 
problem
has been studied for various graph classes~\cite{d,pt,y,w,bi2,d2} with
several papers investigating the \(t\)-tone
chromatic number of graph products~\cite{bi,lmmw, cw} and one 
studying \(t\)-tone coloring of random graphs~\cite{bbdf}. 

The paper is organized as follows.  In Section~\ref{defns-sec} we present our
definitions, and collect some lemmas (proved elsewhere) that we
will use in the remainder of the paper.

In Section~\ref{planar-sec}, we prove a sharp bound on $\tau_2(G)$ for all graphs $G$
with $\mad(G)<12/5$ (which includes planar graphs with girth at least 12), and a
nearly sharp bound on $\tau_2(G)$ for all outerplanar graphs.  
For all planar graphs $G$, we prove a new upper bound on $\tau_2(G)$, that is
sharp up to a factor of $2/\sqrt{3}\approx 1.155$.
We conclude the section with some challenging conjectures.
Our results in Section~\ref{planar-sec} partially answer a question
of West~\cite{west} about $t$-tone coloring of general planar graphs. 

In Section~\ref{cycles-sec} we determine $\tau_t(C_n)$ exactly, for all
$t\in\{3,4,5\}$ and all $n\ge 3$; for each $t$, the value is constant when $n$
is sufficiently large.  The general case relies on a powerful lemma for
combining $t$-tone colorings of subgraphs.  And the stronger lower bounds needed
for some exceptional cases are proved using integer linear programs.
Again, we conclude the section with a challenging conjecture.

In Section~\ref{grids-sec}, we study grid graphs $P_m\square P_n$. We determine
exactly $\tau_3$ and $\tau_4$ and bound $\tau_5$.  

\section{Definitions and Useful Lemmas}
\label{defns-sec}
Let \(G\) be a graph and fix \(v\in V(G)\). We denote by \(N(v)\) the
\textit{neighborhood} of \(v\),  by \(N^2(v)\) the \textit{second
neighborhood} of \(v\) (the set of vertices at distance 2 from \(v\)), by
\(d(v)\) the \textit{degree} of \(v\), and by \(\Delta(G)\) the \textit{maximum
degree} of \(G\). 
We denote by \(\overline{d}(G)\) the \textit{average degree} of \(G\) and by
\(\text{mad}(G)\) the \textit{maximum average degree} of \(G\). 
We write \(H\subseteq G\) if \(H\) is a subgraph of \(G\). 
We let \(P_n\),
\(C_n\), and \(P_m\square P_n\) denote the path on \(n\) vertices, cycle on
\(n\) vertices, and the \(m\times n\) vertex grid graph (where \(\square\)
denotes the Cartesian product). 

For a graph \(G\) and \(t,k\in\ZZ^+\) a \emph{partial \(t\)-tone \(k\)-coloring} of
\(G\) is a function \(f:V(G)\rightarrow \binom{[k]}{t}\cup \emptyset\) such that
\(\card{f(v)\cap f(w)}<d(v,w)\) for all distinct \(v,w\in V(G)\).  To construct
a $t$-tone $k$-coloring of a graph $G$, we will often create a sequence of
partial $t$-tone $k$-colorings, at each step choosing labels for an additional
vertex that was previously unlabeled.

Below we list a number of lemmas that we will use later.  
We generally omit formal proofs, but often include brief proof sketches. The
reader should feel free to skip ahead to Section~\ref{sec3} and only return to this list as needed.
\begin{lemma}
\label{prop: t-tone lower bound from subgraphs}
~\cite[Theorem 11]{fgps} If \(H\) is a subgraph of \(G\), then every \(t\)-tone coloring of \(G\) induces a \(t\)-tone coloring of \(H\). In particular \(\tau_t(H)\leq \tau_t(G)\).
\end{lemma}

\begin{lemma}
\label{prop: t-tone coloring number of C_4}
~\cite[Theorem 1]{w} \(\tau_t(C_4)=4t-2\).
\end{lemma}
\begin{proof}
The labels for each pair of non-adjacent vertices share at most one color. So \(\tau_t(C_4)\geq t\card{V(C_4)}-2(1)\).
\end{proof}

\begin{lemma}
\label{prop:star}
\label{prop: 2-tone star lower bound}
~\cite[Theorem 17]{fgps} All graphs \(G\) satisfy \(
\left\lceil \sqrt{2\Delta(G)+0.25}+2.5\right\rceil\leq \tau_2(G)\).
\end{lemma}
\begin{proof}
The star \(K_{1,\Delta(G)}\) needs \(k\) colors with \(\binom{k-2}{2}\geq \Delta(G)\).
\end{proof}

\begin{lemma}
\label{prop: t-tone coloring paths}
~\cite[Proposition 5]{bp} For all \(t,n\geq 1\) we have \(\tau_t(P_n)=\sum_{i=0}^{n-1}\max{\left\{0,t-\binom{i}{2}\right\}}\). 
\end{lemma}
\begin{proof}
Color the path $v_1\cdots v_n$ in order of increasing subscript.
When vertex \(v_i\) is being colored, for each \(j\in[i-1]\) there
are \(j\) colors used on \(v_{i-j-1}\) that are unused on vertices
closer to \(v_i\). We use these colors on \(v_i\) until either (a) $v_i$ has \(t\)
colors or (b) we run out of vertices. In the latter case, we have used
\(\sum_{j=0}^{i-1}j=\binom{i}{2}\) colors from previous vertices, and need
\(t-\binom{i}{2}\) new colors. When \(\binom{i}{2}\ge t\), no more new colors
are needed.
\end{proof}

\begin{lemma}
\label{lem: 2-tone constant Delta bound}
~\cite[Theorem 2.2]{ckk} Every graph \(G\) satisfies \(\tau_2(G)\leq \left\lceil(2+\sqrt{2})\Delta(G)\right\rceil\). 
\end{lemma}
\begin{proof}
We color greedily avoiding at most \(2\Delta(G)\) colors on neighbors and at most \(\Delta(G)(\Delta(G)-1)\) 2-sets at distance 2. 
\end{proof}

A graph is \(k\)\textit{-degenerate} if each of its subgraphs contains a vertex of degree at most \(k\). 
\begin{lemma}
\label{prop: t-tone coloring k-degenerate graphs}
~\cite[Theorem 3.5]{ckk} If \(G\) is \(k\)-degenerate, \(k\geq 2\), and
\(\Delta(G)\leq r\), then for every \(t\) we have \(\tau_t(G)\leq kt+kt^2\Delta(G)^{1-1/t}\). 
\end{lemma}
\begin{proof}
We color greedily with \(c+kt\) colors. Neighbors forbid at most \(kt\) colors,
and vertices at distance \(d\), for each \(d\in\{2,\dots,t\}\), forbid at most
\(\binom{t}{d}\binom{c-d}{t-d}dk\Delta(G)(\Delta(G)-1)^{d-2}\) sets of size \(t\) that share at least \(d\) elements.
\end{proof}

\begin{lemma}
\label{lem: planar graphs have low degree vertex with at most two high degree
neighbors}
~\cite[Theorem 2]{bkpy} For every planar graph \(G\) there exists \(v\in V(G)\)
such that \(d(v)\leq 5\) and \(v\) has at most two neighbors with degree at least 11.
\end{lemma}

\begin{lemma}
\label{lem: outerplanar graphs have a very low degree vertex with low degree
neighbors}
~\cite[Theorem 5]{f} For every outerplanar graph \(G\) there exists \(xy\in E(G)\) with \(d(x)=1\), or \(d(x)=2\) and \(d(y)\leq 4\).
\end{lemma}

We conclude this section with a construction of a planar graph that improves the
trivial lower bound, from Lemma~\ref{prop:star}, on colors needed to 2-tone color a
planar graph of given maximum degree.

\begin{lemma}
\label{lem:Ht}
\label{lem: 2-tone coloring the fat triangle}
 For each \(t\geq 1\), we form \(H_t\) from \(K_3\) by replacing each edge
\(vw\in E(K_3)\) with a copy of \(K_{2,t}\), identifying the high degree
vertices with \(v\) and \(w\). For all $t$ we have
$\left\lceil\sqrt{3\Delta(H_t)+0.25}+0.5\right\rceil
\leq\tau_2(H_t)\leq
\left\lceil\sqrt{3\Delta(H_t)+30.25}+0.5\right\rceil$.
(When $t\ge 33$ these two bounds differ by at most 1.)

\end{lemma}
\begin{proof}
Fix a positive integer $k$ to be determined later.
We consider a 2-tone $k$-coloring of $H_t$.
It is easy to check that $\tau_2(C_6)=5$, so assume $t\ge 2$.
Let $x$, $y$, and $z$ denote the vertices of degree at least 4.
For the lower bound, note that all $3\Delta(H_t)/2=3t$ vertices excluding $x$,
$y$, $z$ must get distinct 2-element subsets of $\binom{k}{2}$.
The inequality $\binom{k}2\ge 3\Delta(H_t)/2$ is equivalent to the lower bound.

Now we prove the upper bound.
Color $x$ with $\{1,2\}$; color $y$ with $\{3,4\}$; and color $z$ with
$\{5,6\}$.  Now we assign each remaining vertex of $H_t$ a distinct element of
$\binom{[k]}2\setminus\{\{1,2\},\{3,4\},\{5,6\}\}$.  This requires that no
vertex of degree 2 receive a label from $\binom{[6]}2$.  Thus, we need
$\binom{k}2-\binom{6}2\ge 3t$.  This inequality is equivalent to the upper
bound.  We must also ensure that the coloring is proper, i.e., all labels
including 1 or 2 (other than $\{1,2\}$) be used on vertices non-adjacent to $x$,
and similarly for $\{3,4\}$ with $y$ and for $\{5,6\}$ with $z$.
However, this is easy to ensure.

Finally, we show that the bounds differ by at most 1 when $t\ge 33$.
For this conclusion, it suffices that
$\sqrt{3\Delta(H_t)+30.25}-\sqrt{3\Delta(H_t)+0.25}\le 1$.  This inequality
holds when $\Delta(H_t)\ge 70$, i.e., when $t\ge 35$.  And it easy to check the remaining 2 cases by hand.
\begin{figure}[H]
\centering
\begin{tikzpicture}
\node[fill=black, circle, inner sep=2pt, label={180:{\(x\)}}] (x) at (0,0) {};
\node[fill=black, circle, inner sep=2pt, label={0:{\(y\)}}] (y) at (3,0) {};
\node[fill=black, circle, inner sep=2pt, label={90:{\(z\)}}] (z) at (1.5,2.5) {};

\node[Simple Node] (x0) at (0:1.5cm) {}; \draw[thick,black] (x)--(x0)--(y);
\node[Simple Node] (x1) at (10:1.52cm) {}; \draw[thick,black] (x)--(x1)--(y);
\node[Simple Node] (x2) at (-10:1.52cm) {}; \draw[thick,black] (x)--(x2)--(y);

\node[Simple Node] (z0) at (60:1.5cm) {}; \draw[thick,black] (x)--(z0)--(z);
\node[Simple Node] (z1) at (70:1.52cm) {}; \draw[thick,black] (x)--(z1)--(z);
\node[Simple Node] (z2) at (50:1.52cm) {}; \draw[thick,black] (x)--(z2)--(z);

\begin{scope}[xshift=3cm]
\node[Simple Node] (y0) at (120:1.5cm) {}; \draw[thick,black] (y)--(y0)--(z);
\node[Simple Node] (y1) at (110:1.52cm) {}; \draw[thick,black] (y)--(y1)--(z);
\node[Simple Node] (y2) at (130:1.52cm) {}; \draw[thick,black] (y)--(y2)--(z);
\end{scope}

\end{tikzpicture}
\caption{The graph \(H_3\).}
\label{fig: fat triangle for t=3}
\end{figure}
\aftermath
%
\end{proof}

\section{2-tone Coloring of Planar Graphs}
\label{sec3}
\label{planar-sec}
In this section, we prove our first two main results. In
Theorem~\ref{thm:outerplanar} we
determine $\tau_2(G)$ for all outerplanar graphs, up to a small additive constant.
And in Theorem~\ref{thm:mad} we determine $\tau_2(G)$ for all graphs $G$ with
$\mad(G)<12/5$ and $\Delta(G)\ge 11$.  This includes planar graphs with girth
at least 12.  As a warm-up, in Theorem~\ref{thm:planar} we bound $\tau_2(G)$
for all planar graphs; as $\Delta(G)$ grows, our bound is sharp asymptotically
up to a factor of $2/\sqrt{3}\approx 1.155$.

All our proofs in this section proceed by minimal counterexample.  This approach
requires extra care, since a 2-tone coloring of a subgraph $H$ of $G$ might fail
to induce a 2-tone coloring of $G[V(H)]$.
Specifically, if we delete a vertex $v$ to form a subgraph $H$, we allow the
possibility that neighbors of $v$ in $G$ will receive identical labels in $H$;
of course, this is forbidden in a 2-tone coloring of $G$.  To avoid this difficulty, rather than
deleting vertices, we often instead contract edges, which never increases
distances.  However, this adds the potential issue of increasing the maximum
degree.  To avoid this pitfall, we typically contract an
edge with one endpoint of degree at most 2.
To extend a partial 2-tone coloring of a graph $G$, we will often use the following
helpful lemma.

\begin{lemma}
\label{lem: 2-tone coloring via 1st and 2nd neighborhoods}
Let \(G\) be a graph and \(\varphi\) be a partial \(2\)-tone \(k\)-coloring of \(G\). For any uncolored vertex \(v\in V(G)\), if \(\binom{k-2\card{N(v)}}{2} > \card{N^2(v)}\), then \(\varphi\) can be extended to \(v\).
\end{lemma}
\begin{proof}
Let \(G\), \(\varphi\), and \(v\) be as in the lemma. To extend \(\varphi\) to
\(v\), we must avoid all colors used on \(N(v)\), which forbids at most
\(2\card{N(v)}\) colors. 
We must also avoid all 2-sets used on \(N^2(v)\), which forbids at most
$|N^2(v)|$ 2-sets.
Thus, it suffices to have 
\(\binom{k-2\card{N(v)}}{2} > \card{N^2(v)}.\) 
\begin{figure}[h]
\centering
\begin{tikzpicture}
\node[fill=black, circle, inner sep=2pt, label={180:{\(v\)}}] (v) at (0,0) {};
    \foreach \i in {0,2,3,4,5} {
        \node[Simple Node] (v\i) at (\i*18:1.5cm) {}; 
        \draw[thick, black] (v)--(v\i) node[pos=0.9] (\i) {};
        \foreach \j in {0,1} {
            \tikzset{shift={(\i*18:1.5cm)}}
            \ifnum \i=0 
            \node[Simple Node] (w\i\j) at (\j*30+\i*18:\j*2+1 cm) {}; 
            \draw[thick, black] (v\i)--(w\i\j) {};
            \else 
            \node[Simple Node] (w\i\j) at (-\j*30+\i*18:\j*2+1 cm) {}; 
            \draw[thick, black] (v\i)--(w\i\j) {};
            \fi
        }
        \draw[thick, black, dotted] (w\i0) -- (w\i1);
    }
\draw[dotted, black,thick] (v0) -- (v2); \draw[dotted, black, thick] (w01) -- (w21);

\node at (-1,1.5) {\(\textcolor{black}{N(v)}\)};
\begin{pgfonlayer}{pre main}
    \draw[thick, draw=black, name path = A1] (1.85,0) arc (0:90:1.85cm);
    \draw[thick, black, name path = B1] (1.15,0) arc (0:90:1.15cm);
    \tikzfillbetween[of=A1 and B1]{black, opacity=0.2};
    \filldraw[thick, black, fill=black, fill opacity=0.2] (1.85,0) arc (0:-180:0.35cm);
    \filldraw[thick, black, fill=black, fill opacity=0.2] (0,1.85) arc (90:270:0.35cm);
\end{pgfonlayer}

\node at (-1,2.5) {\(\textcolor{black}{N^2(v)}\)};
\begin{pgfonlayer}{pre main}
    \draw[thick, draw=black, name path = A2] (2.85,0) arc (0:90:2.85cm);
    \draw[thick, black, name path = B2] (2.15,0) arc (0:90:2.15cm);
    \tikzfillbetween[of=A2 and B2]{black, opacity=0.2};
    \filldraw[thick, black, fill=black, fill opacity=0.2] (2.85,0) arc (0:-180:0.35cm);
    \filldraw[thick, black, fill=black, fill opacity=0.2] (0,2.85) arc (90:270:0.35cm);
\end{pgfonlayer}

\node at (-1,3.5) {\(\textcolor{black}{G-\{N(v)\cup N^2(v)\cup \{v\}\}}\)};
\draw[black, thick, name path=A3] plot [smooth, tension=0.8] coordinates {(0,2.85) (1,3.1) (2.35,2.35) (3.1,1) (2.85,0)};
\tikzfillbetween[of=A2 and A3]{black, opacity=0.2};

\end{tikzpicture}
\caption{A vertex \(v\) with its neighbours and second neighbours.}
\label{fig: 2-tone coloring with neighbourhoods and second neighbourhoods}
\end{figure}
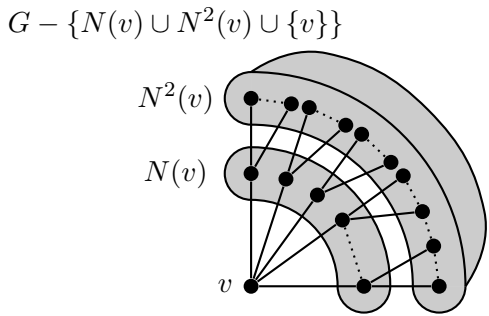
\end{proof}

We first prove an upper bound on $\tau_2(G)$ for every planar graph $G$, and
then show how to strengthen it for two classes of ``sparse'' planar graphs.
For a general planar graph $G$ (with maximum degree $\Delta(G)$), our upper bound
in the next theorem differs from the lower bound in Lemma~\ref{prop: 2-tone star lower bound}
by a factor of approximately $\sqrt{2}$. 
However, for our construction $H_t$ in Lemma~\ref{lem:Ht} the present upper bound
differs from the lower bound by only a facor of $2/\sqrt{3}\approx 1.155$.

\begin{theorem}
\label{thm:planar}
\label{thm: 2-tone coloring planar graphs}
 If \(G\) is a planar graph, then \(\tau_2(G)\leq
\left\lfloor\sqrt{4\Delta(G)+50.25}+31.1\right\rfloor\leq
\left\lfloor\sqrt{4\Delta(G)}+36.5\right\rfloor\). Furthermore,
\(\tau_2(G)\leq\max\left\{41, \left\lfloor\sqrt{4\Delta(G)+50.25}+11.5\right\rfloor\right\}.\)
 \end{theorem}
 \begin{proof}
In the first statement, the second inequality is easy to verify, so we focus on
the first.
The second statement is clearly stronger when $\Delta(G)$ is sufficiently large, but we
include the first to give a better bound when $\Delta(G)$ is small.
We prove both statements simultaneously.

 Suppose the theorem is false and let \(G\) be a counterexample that minimizes
\(\card{V(G)}\). If \(\Delta(G)\leq 12\), then Lemma~\ref{lem: 2-tone constant
Delta bound} gives \(\tau_2(G)\leq
\left\lceil(2+\sqrt{2})\Delta(G)\right\rceil\leq
\left\lfloor\sqrt{4\Delta(G)+50.25}+31.1\right\rfloor\leq 41\). So we assume
that \(\Delta(G)\geq 13\). By Lemma~\ref{lem: planar graphs have low degree
vertex with at most two high degree neighbors}, there exists \(v\in V(G)\) such
that \(d(v)\leq 5\) and \(v\) has at most two neighbors of degree at least 11.
If \(d(v)\geq 3\), then pick \(w\in N(v)\) with \(d(w)\leq 10\); otherwise let
\(w\) be an arbitrary neighbor of \(v\). Form \(H\) from \(G\) by contracting \(vw\). Since \(\card{V(H)}<\card{V(G)}\) and \(\Delta(H)\leq \max\{\Delta(G),5+10-2\}=\Delta(G)\), by induction \(\tau_2(H)\leq
\max\left\{41,\left\lfloor\sqrt{4\Delta(H)+50.25}+11.5\right\rfloor\right\}\leq
\max\left\{41,\left\lfloor\sqrt{4\Delta(G)+50.25}+11.5\right\rfloor\right\}\). 
Similarly, \(\tau_2(H)\leq
\left\lfloor\sqrt{4\Delta(H)+50.25}+31.1\right\rfloor\leq
\left\lfloor\sqrt{4\Delta(G)+50.25}+31.1\right\rfloor\). 
Now
\(N_G(v)\) forbids at most \(2\card{N_G(v)}\leq 10\) colors from use on \(v\).
Further, vertices in \(N_G^2(G)\) forbid at most
\(2(\Delta(G)-1)+3(9)=2\Delta(G)+25\) distinct 2-sets from use on \(v\). By
Lemma~\ref{lem: 2-tone coloring via 1st and 2nd neighborhoods} we can extend
any \(2\)-tone \(k\)-coloring of \(H\) to a 2-tone \(k\)-coloring of \(G\)
whenever \(\Delta(G)\geq 13\) and
 \[\binom{k-10}{2}> 2\Delta(G)+25.\]
This inequality is easy to verify when
\(k=\left\lfloor\sqrt{4\Delta(G)+50.25}+11.5\right\rfloor\), 
which completes the proof of both statements.
 \end{proof}

%

In the next two theorems, we consider special classes of planar graphs that are
in a sense ``tree-like". For these graphs, we improve the leading coefficient in
the bound of Theorem~\ref{thm: 2-tone coloring planar graphs} by a factor of
approximately $\sqrt{2}$, so that it matches that in the lower bound
given by Lemma~\ref{prop: 2-tone star lower bound}.
 
 \begin{theorem}
 \label{thm:outerplanar}
 \label{thm: 2-tone coloring outerplanar graphs}
 If \(G\) is outerplanar, then
\(\tau_2(G)\leq\left\lfloor\sqrt{2\Delta(G)+4.25}+5.5\right\rfloor\leq
\left\lfloor\sqrt{2\Delta(G)}+6.6\right\rfloor\).
 \end{theorem}
 \begin{proof}
The second inequality is easily verified by algebra, so we focus on the first.
 Suppose the theorem is false and let \(G\) be a counterexample minimizing \(\card{V(G)}\). 
Note that the class of outerplanar graphs is closed under edge contraction.

By Lemma~\ref{lem: outerplanar graphs have a very low degree vertex with low
degree neighbors} there exists \(vw\in E(G)\) such that
\(d(v)=1\), or \(d(v)=2\) and \(d(w)\leq 4\). In either case, form \(H\) by
contracting \(vw\) (restricting to the underlying simple graph if we create a
pair of parallel edges). Note that \(\card{H}<\card{G}\) and \(\Delta(H)\leq
\Delta(G)\). By the minimality of \(G\), 
\[\tau_2(H)\leq\left\lfloor\sqrt{2\Delta(H)+4.25}+5.5\right\rfloor\leq
\left\lfloor\sqrt{2\Delta(G)+4.25}+5.5\right\rfloor.\] 
The vertices in \(N_G(v)\) forbid at most \(2\card{N(v)}\leq 4\) colors from use
on \(v\). Further, the vertices in \(N_G^2(v)\) forbid at most
\(\Delta(G)-1+(4-1)=\Delta(G)+2\) distinct 2-sets from use on \(v\). By
Lemma~\ref{lem: 2-tone coloring via 1st and 2nd neighborhoods} we can extend any \(2\)-tone \(k\)-coloring of \(H\) to \(G\) when
\[\binom{k-4}{2}> \Delta(G)+2.\]
This inequality is easy to verify when \(k=\left\lfloor\sqrt{2\Delta(G)+4.25}+5.5\right\rfloor\).
\end{proof}

Lemma~\ref{lem: mad thread lemma} is a structural result that we will
use to prove Theorem~\ref{thm: 2-tone coloring girth 12 graphs}.
As a special case, that theorem will exactly determine $\tau_2$
for planar graphs with sufficiently large girth and max degree.

We will also need some new definitions. 
A \(d^+\)-vertex,
\(d^-\)-vertex, or \(d\)-vertex is, respectively, a vertex of degree at least
\(d\), at most \(d\), and exactly \(d\).
An \emph{\(\ell\)-thread} in a graph \(G\) is a trail of length \(\ell + 1\) in \(G\)
whose \(\ell\) internal vertices have degree 2 in \(G\). We
refer to the non-internal vertices of an \(\ell\)-thread as \emph{endpoints}. So an
\(\ell\)-thread has two endpoints, not necessarily distinct. 
For  Lemma~\ref{lem: mad thread lemma} and Theorem~\ref{thm: 2-tone coloring
girth 12 graphs} we present the proofs as if each \(\ell\)-thread has two
distinct endpoints, but all arguments remain valid if the endpoints are not
distinct.

\begin{lemma}
\label{lem:thread}
\label{lem: mad thread lemma}
Let \(G\) be a graph with \(\delta(G)\geq 2\). If \(\mad(G)< 12/5\), then \(G\) contains at least one of the following:
\begin{itemize}
    \item[(a)] a \(4\)-thread,
    \item[(b)] a \(3\)-thread with a \(5^-\)-vertex as an endpoint, or
    \item[(c)] a \(2\)-thread with a \(3^-\)-vertex and a \(5^-\)-vertex as endpoints.
\end{itemize} 
\end{lemma}
\begin{proof}
Let \(G\) be a graph with \(\delta(G)\geq 2\) and \(\mad(G)< 12/5\). 
Assume for contradiction that \(G\) has no threads of type (a), (b), and (c). 
If $G$ contains a 2-regular component, then it contains an instance of (c); so
assume no component of $G$ is 2-regular.  Thus, every 2-vertex appears in a
unique maximal thread, and the endpoints of that thread are $3^+$-vertices.
We give each vertex \(v\) initial charge \(d(v)\). To redistribute charge,
each maximal thread 
takes charge \(1-12/(5d(v))\) from each of its endpoints. 
Since \(G\) has no \(4\)-thread, each maximal
thread has at most \(3\) internal vertices. 
If a thread $t$ has a vertex $v$ as an endpoint, then the charge that $t$
receives from $v$ is: $1-12/(3(5))=1/5$ if $d(v)=3$; and $\ge1-12/(4(5))=2/5$ if
$d(v)\ge 4$; and $\ge1-12/(6(5))=3/5$ if $d(v)\ge 6$.

Each \(1\)-thread gains at least \(1/5\) 
from each endpoint, 
so finishes with at least $12/5$.

Each \(2\)-thread cannot be an instance of (c), so either (i) both of its
endpoints are \(4^+\)-vertices or (ii) it has a \(6^+\)-vertex as an endpoint.
So a \(2\)-thread gains either (i) at least \(2/5\) from each endpoint or  (ii)
at least \(3/5\) from the endpoint that is a \(6^+\)-vertex and at least \(1/5\) from 
the other endpoint. Thus, each 2-thread finishes with at least
$2(2)+4/5=2(12/5)$.

Each \(3\)-thread has a \(6^+\)-vertex for each endpoint, otherwise \(G\)
contains (b). So a \(3\)-thread gains at least \(3/5\) from each endpoint. 
Thus, each 3-thread finish with at least $3(2)+6/5=3(12/5)$.
If \(v\) is an endpoint of a thread, then \(v\) sees at most \(d(v)\) threads. 
Thus, $v$ has final charge $d(v)-d(v)(1-12/(5d(v))=12/5$.
This implies that \(\overline{d}(G) \geq 12/5\); which contradicts the
hypothesis \(\mad(G)< 12/5\).
\end{proof}

\begin{theorem}
\label{thm:mad}
\label{thm: 2-tone coloring girth 12 graphs}
If \(G\) is a graph with \(\mad(G)< 12/5\), then
\(\tau_2(G)\le\max\left\{7,\left\lceil\sqrt{2\Delta(G)+0.25}+2.5\right\rceil\right\}\). 
Further, if \(G\) is planar with girth at least 12 and $\Delta(G)\geq
7$, then \(\tau_2(G)=\left\lceil\sqrt{2\Delta(G)+0.25}+2.5\right\rceil\).
\end{theorem}
\begin{proof}
The second statement follows from the first since a planar graph \(G\) with
girth at least 12 has \(\mad(G)<2(12)/(12-2)=12/5\) and 
Lemma~\ref{prop: 2-tone star lower bound} implies that if \(\Delta(G)\geq 7\), then \(\tau_2(G)\geq 7\). 
We now prove the first statement.

Suppose the theorem is false and let \(G\) be a counterexample minimizing
\(\card{V(G)}\). 
If there exists $v$ with $d(v)\le 1$, then by minimality $\tau_2(G-v)\le\max
\left\{7,\left\lceil\sqrt{2\Delta(G)+0.25}+2.5\right\rceil\right\}$. 
And by Lemma~\ref{lem: 2-tone coloring via 1st and 2nd neighborhoods} we get
\(\tau_2(G)\leq\max\left\{7,\left\lceil\sqrt{2\Delta(G)+0.25}+2.5\right\rceil\right\}\). 
Thus, we assume $\delta(G)\ge 2$.

By Lemma~\ref{lem: mad thread lemma} we know \(G\) contains configuration (a),
(b), or (c) in that lemma.  We will show that none of these configurations
can appear in our minimal counterexample $G$. To do so, we form a
subgraph \(H\) by deleting some vertices of \(G\), color \(H\) by
minimality, and extend our coloring of $H$ to the deleted vertices of \(G\), 
to contradict that \(G\) was a counterexample. 
Let $k_G=\max\left\{7,\left\lceil\sqrt{2\Delta(G)+0.25}+2.5\right\rceil\right\}$.
For an arbitrary subgraph $H$ of $G$ (which will be clear from context), 
let $k_H=\max\left\{7,\left\lceil\sqrt{2\Delta(H)+0.25}+2.5\right\rceil\right\}$.


\textbf{Case 1: $\bm{G}$ contains a 4-thread, as shown in
Figure~\ref{fig:4-thread reducible configuration}.} Form \(H\) from $G$ by
deleting \(v_2\) and \(v_3\). Note that \(\card{H}<\card{G}\) and
\(\Delta(H)\leq\Delta(G)\). By the minimality of \(G\),
we have $\tau_2(H)\le k_H\le k_G$.
Let \(\varphi\) be a
\(2\)-tone $k_G$-coloring of \(H\). 
By Lemma~\ref{lem: 2-tone coloring via 1st and 2nd neighborhoods},
since $k_G\ge 7$ we can extend \(\varphi\) to \(v_1\) followed by \(v_2\),
a contradiction.

\begin{figure}[H]
\centering
\begin{tikzpicture}
\node[uStyle, label={270:{\(x\)}}] (v0) at (0,0) {}; 
\foreach \i in {-1,1}{ 
\node[uStyle] (x\i) at (\i*140:1cm) {};\draw[thick,black] (x\i)--(v0) node[pos=0] (q\i) {};}
\draw[dotted,black,thick] (q1)--(q-1);
\begin{scope}[xshift=5cm] 
    \node[uStyle, label={270:{\(y\)}}] (v5) at (0,0) {};
    \foreach \i in {-1,1}{
    \node[uStyle] (y\i) at (\i*40:1cm) {};\draw[thick,black] (y\i)--(v5) node[pos=0] (r\i) {};}
    \draw[dotted,black,thick] (r1)--(r-1);
\end{scope}
\foreach \i in {1,2,3,4}{\node[uStyle, label={270:{\(v_{\i}\)}}] (v\i) at (\i,0) {};}
\foreach \i [evaluate=\i as \x using \i-1] in {1,2,3,4,5}{\draw[thick, black] (v\x)--(v\i);}
\end{tikzpicture}
\caption{A \(4\)-thread with endpoints $x$ and $y$.
\label{fig:4-thread reducible configuration}
} 
\end{figure}
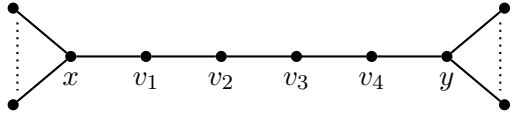

\textbf{Case 2: \(\bm{G}\) contains a 3-thread, as shown in Figure~\ref{fig:3-thread
reducible configuration}.} Form \(H\) by deleting \(v_2\) and \(v_3\). Note that
\(\card{H}<\card{G}\) and \(\Delta(H)\leq\Delta(G)\). By the minimality of \(G\),
we have $\tau_2(H)\le k_H\le k_G$.
Let \(\varphi\) be a
2-tone $k_G$-coloring of \(H\). 
By Lemma~\ref{lem: 2-tone coloring via 1st and 2nd neighborhoods},
since $k_G\ge 7$ we can extend \(\varphi\) to \(v_2\) followed by \(v_3\),
a contradiction.
In particular, since \(y\) forbids
2 colors from use on \(v_3\) and the vertices at distance 2 from \(v_3\)
forbid at most 4 distinct 2-sets from use on \(v_3\), since $k_G\ge 7$ we have
at least $\binom{5}2-5=5$ remaining 2-sets available for \(v_3\).
Afterwards, it is easy to color $v_2$.  This finishes the extension of $\vph$ to
a 2-tone $k_G$-coloring of $G$, which is a contradiction.


\begin{figure}[H]
\centering
\begin{tikzpicture}
\node[uStyle, label={270:{\(x\)}}] (v0) at (0,0) {}; 
\foreach \i in {-1,1}{ 
\node[uStyle] (x\i) at (\i*140:1cm) {};\draw[thick,black] (x\i)--(v0) node[pos=0] (q\i) {};}
\draw[dotted,black,thick] (q1)--(q-1);
\begin{scope}[xshift=4cm] 
    \node[uStyle, label={270:{\(y\)}}] (v4) at (0,0) {};
    \foreach \i in {0,1,2,3}{
    \node[uStyle] (y\i) at (\i*30-1.5*30:1cm) {};\draw[thick,black] (y\i)--(v4);}
\end{scope}
\foreach \i in {1,2,3}{\node[uStyle, label={270:{\(v_{\i}\)}}] (v\i) at (\i,0) {};}
\foreach \i [evaluate=\i as \x using \i-1] in {1,2,3,4}{\draw[thick, black] (v\x)--(v\i);}
\end{tikzpicture}
\caption{A \(3\)-thread with endpoints $x$ and $y$, where $d(y)\leq 5$.
\label{fig:3-thread reducible configuration}
}
\end{figure}
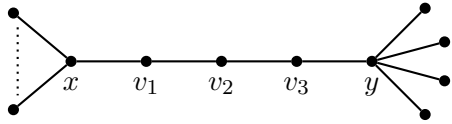

\textbf{Case 3: \(\bm{G}\) contains a 2-thread, as shown in
Figure~\ref{fig:2-thread reducible configuration}.} 
Form \(H\) from $G$ by deleting \(v_1\). Note that \(\card{H}<\card{G}\) and
\(\Delta(H)=\Delta(G)\). By the minimality of \(G\), we have 
$\tau_2(H)\leq k_H\le k_G$
Let \(\varphi\) be a \(2\)-tone $k_G$-coloring of \(H\). 
Note that \(\varphi\) might fail to induce a partial \(2\)-tone coloring of
\(G\) since it is possible that \(\varphi(v_2)=\varphi(x)\), which creates a
problem since \(d(v_2,x)=2\).
To avoid this issue we can simply re-color \(v_2\), since $d(v_2)=2$. In this
case, \(v_2\) is a leaf of \(H\), so its neighbor forbids \(2\) colors from use
on \(v_1\); furthermore, the vertices at distance 2 from \(v_2\) forbid at most
5 distinct 2-sets from use on \(v_1\). So we can recolor \(w\) with
another 2-set, since
$\binom{7-2}{2}> 4+1$; in fact, we have at least 4 choices of label for $v_2$.
Thus, we assume that \(\varphi\) induces a proper 2-tone coloring of \(G\).
Finally, we consider coloring $v_1$.
Its two neighbors forbid at most $2(2)=4$ colors.  And the three vertices at
distance two forbid an additional three 2-sets.  If $k_G\ge 8$, then we have a
2-set available to use on $v_1$.  So assume instead that $k_G=7$.
If no 2-sets are available to use on $v_1$, then the two 2-sets used on its
neighbors are disjoint.  Further, the three 2-sets used on vertices at distance
two are distinct, and they are all disjoint from the set of colors used on its
neighbors.  But now to escape this situation we can recolor $v_2$ with one of
the other 4 possible 2-sets we had to choose from.  Afterward, we can extend the
2-tone 7-coloring to $G$, a contradiction.

\begin{figure}[H]
\centering
\begin{tikzpicture}
\node[uStyle, label={270:{\(x\)}}] (v0) at (0,0) {}; 
\foreach \i in {-1,1}{ 
\node[uStyle] (x\i) at (180-\i*30:1cm) {};\draw[thick,black] (x\i)--(v0);}
\begin{scope}[xshift=3cm] 
    \node[uStyle, label={270:{\(y\)}}] (v3) at (0,0) {};
    \foreach \i in {0,1,2,3}{
    \node[uStyle] (y\i) at (\i*30-1.5*30:1cm) {};\draw[thick,black] (y\i)--(v3);}
\end{scope}
\foreach \i in {1,2}{\node[uStyle, label={270:{\(v_{\i}\)}}] (v\i) at (\i,0) {};}
\foreach \i [evaluate=\i as \x using \i-1] in {1,2,3}{\draw[thick, black] (v\x)--(v\i);}
\end{tikzpicture}
\caption{A \(2\)-thread with endpoints $x$ and $y$, where $d(x)=3$ and $d(y)\leq
5$.
\label{fig:2-thread reducible configuration}
}
\end{figure}
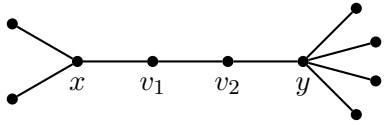
\aftermath
\end{proof}

We conclude this section with a few conjectures.

\begin{conj}
There exists a constant $C$ such that all planar $G$ satisfy
$\tau_2(G)\le \sqrt{3\Delta(G)}+C$.
\end{conj}
Perhaps the following stronger statement holds.
It is essentially best possible, due to Lemma~\ref{lem:Ht}.

\begin{conj}
Every planar graph $G$ satisfies
$\tau_2(G)\le \left\lceil\sqrt{3\Delta(G)+30.25}+0.25\right\rceil$.
\end{conj}

We also believe that for planar graphs the girth requirement in
Theorem~\ref{thm:mad} can be significantly weakened.

\begin{conj}
There exists a constant $C$ such that every planar graph $G$ with girth at least
5 satisfies $\tau_2(G)\le \sqrt{3\Delta(G)}+C$.
\end{conj}

It is interesting to note the following.  For every integer $t\ge 2$ there
exists a girth $g_t$ and a maximum degree $\Delta_t$ such the maximum value of
$\tau_t(G)$, taken over all planar graphs $G$ with girth at least $g_t$ and
$\Delta(G)\ge \Delta_t$, is achieved by a tree.  Cranston, Kim, and
Kinnersley~\cite[Theorem~2]{ckk} showed that this maximum (for trees) is bounded by
$c_t\sqrt{\Delta(G)}$ for some constant $c_t$; and this is asymptotically sharp.
We briefly sketch the extension to planar graphs with sufficiently large girth
and maximum degree.  Following an approach similar to (but simpler than) the
proof of Lemma~\ref{lem:thread}, we can prove that if $G$ has sufficiently low maximum
average degree, then it contains either a $1^-$-vertex or a $3t$-thread.
Every $1^-$-vertex can be handled inductively (by coloring greedily).
For a $3t$-thread, we delete the middle $t$ vertices and color the smaller graph
by induction.  We choose $\Delta_t$ large enough that 
$\tau_t(K_{1,\Delta_t})\ge 3\tau_t(P_t)$.  (Recall that
$\tau_t(K_{1,\Delta_t})\ge \tau_2(K_{1,\Delta_t})\ge \sqrt{2\Delta_t}$, by
Lemma~\ref{prop:star}.) Now the number of colors forbidden on
all of the uncolored vertices (taken together) is at most $2\tau_t(P_t)$.
Thus, we have at least $\tau_t(P_t)$ colors that are available for use on all of
the uncolored vertices.  So we can extend the coloring.
\section{3-Tone, 4-Tone, and 5-Tone Coloring of Cycles}
\label{cycles-sec}
We can easily prove that 
$\tau_t(C_n)=O(t^{3/2})$, as follows.  Let $f(t):=\tau_t(P_t)$. 
By Lemma~\ref{prop: t-tone coloring paths}, there exists a constant $c$
such that $\tau_t(P_t)\le ct^{3/2}$ for all $t$.  Further,
$\tau_t(P_n)=\tau_t(P_t)$ for all $n\ge t$. 
Whenever \(n\geq 2t+2\), to prove \(\tau_t(C_n)\leq
2f(t)\) we simply color the first \(t+1\) vertices with one set of \(f(t)\)
colors and the remaining vertices with a disjoint set of \(f(t)\) colors. 
But is it true that $\tau_t(C_n)=\tau_t(P_n)$ for all $n$
sufficiently large (as a function of $t$)?
Bickle and Phillips~\cite[Theorem 18]{bp} showed that 
$\tau_2(C_n)=6$ when $n\in\{3,4,7\}$ and otherwise
\(\tau_2(C_n)=\tau_2(P_n)=5\).
We generalize their approach to prove analogous results for $\tau_3$, $\tau_4$,
and $\tau_5$.  Our next lemma plays a key role in these proofs.

\begin{lemma}
\label{lem: concatenating colored small cycles to color cycles}
Fix $t,k,n\in\ZZ^+$. Let $\C$ be a set of positive integers, each at least $t$.
If $n$ can be written as an integer linear combination of elements in $\C$ (with
nonnegative coefficients), then $\tau_t(C_n)\le k$ provided that the following
two properties hold:

\begin{itemize}
    \item[(1)] For each $\ell\in \C$, there exist a $t$-tone $k$-coloring
$\vph_{\ell}$ of $C_{\ell}$; and
\item[(2)] For each ordered pair $(\ell_1,\ell_2)\in \C\times\C$ (allowing
$\ell_1=\ell_2$), we get a $t$-tone $k$-coloring of $C_{2t}$ if we color its
first $t$ vertices as vertices $\ell_1-t+1,\ldots, \ell_1$ of $C_{\ell_1}$ under
$\vph_{\ell_1}$ and we colors its last $t$ vertices as vertices $1,\ldots, t$ of
$C_{\ell_2}$ under $\vph_{\ell_2}$.
\end{itemize}
\end{lemma}

\begin{proof}
Fix $t$, $k$, and $\C$ satisfying the hypotheses.  We prove the stronger
statement that if $n$ satisfies the hypotheses, then $C_n$ has a $t$-tone
$k$-coloring in which its vertices are partitioned into copies of $P_{\ell_i}$,
with each $\ell_i\in \C$, and each copy of $P_{\ell_i}$ colored by
$\vph_{\ell_i}$.  Our proof is by induction on the sum of the coefficients in
the integer linear combination representation of $n$.  

Assume, by symmetry, that $\ell_1$ has a positive coefficient, and let
$n':=n-\ell_1$.  By hypothesis, we have the desired $t$-tone $k$-coloring
$\vph_{n'}$ of $C_{n'}$.  We insert a path on $\ell_1$ vertices between the
``first'' and ``last'' vertex of the cycle $C_{n'}$ to get $C_n$.  Note that
$\vph_{n'}$ induces a partial $t$-tone $k$-coloring of $C_n$, with these
$\ell_1$ successive vertices uncolored.  To extend this partial coloring, we
color the uncolored vertices using
$\vph_{\ell_1}$.  By properties (1) and (2), this yields a $t$-tone $k$-coloring
of $C_n$, as desired.
%
%
\end{proof}

Note that Property (2) holds trivially if each $t$-tone coloring $\vph_{\ell_i}$
agrees on (is identical on) its first $t$ vertices. For example, in Figure~\ref{fig: t-tone coloring cycles example}, we can use Lemma~\ref{lem: concatenating colored small cycles to color cycles} with \(\C =\{4,5\}\) to show \(\tau_3(C_{13})\leq 10\) since \(13=2(4)+1(5)\) and the \(3\)-tone \(10\)-colorings of \(C_4\) and \(C_5\) agree in the first \(3\) vertices. 

\begin{figure}[H]
\centering
\begin{tikzpicture}[thick]
    \foreach \i in {0,1,2,3}{
        \node[uStyle,fill=black,draw=black,label={45+90*\i:
        \ifnum \i=0 \(\{4,5,6\}\) \fi
        \ifnum \i=1 \(\{1,2,3\}\) \fi
        \ifnum \i=2 \(\{4,9,10\}\) \fi
        \ifnum \i=3 \(\{1,7,8\}\) \fi
        }] (v\i) at (45+90*\i:1cm) {};
    }
    \foreach \i/\j in {0/1,1/2,2/3,3/0}{
        \draw[thick, black, decorate, decoration={zigzag, segment length=1mm, amplitude=0.5mm}] (v\i) -- (v\j);
    }
    \begin{scope}[yshift=-3.5cm]
        \foreach \i in {0,1,2,3,4}{ 
            \node[uStyle,fill=black, draw=black, label={18+72*\i:
                \ifnum \i=0 \(\{1,7,8\}\) \fi
                \ifnum \i=1 \(\{4,5,6\}\) \fi
                \ifnum \i=2 \(\{1,2,3\}\) \fi
                \ifnum \i=3 \(\{5,7,10\}\) \fi
                \ifnum \i=4 \(\{4,2,9\}\) \fi
                }] (w\i) at (18+72*\i:1cm) {};}
        \foreach \i/\j in {0/1,1/2,2/3,3/4,4/0}{
                \draw[thick, black, decorate, decoration={snake,segment length=3.4mm}] (w\i) -- (w\j); 
        }
    \end{scope}
    \begin{scope}[yshift=-2cm, xshift=7.7cm]
        \foreach \i/\j in {0/0,1/1,2/2,3/3,4/4,5/7.2,6/6.2,7/5.2,8/4.2,9/3,10/2,11/1,12/0}{
            \ifnum \i=0 
                \node[uStyle,fill=black,draw=black, label={\i*360/13:\(\{5,7,10\}\)}] (u\i) at (34.61+\i*360/13:2.25cm) {};
            \else 
                \ifnum \i<5 \node[uStyle, fill=black,draw=black, label={34.61+\i*360/13:
                    \ifnum \i=1 \(\{4,2,9\}\) \fi
                    \ifnum \i=2 \(\{1,7,8\}\) \fi
                    \ifnum \i=3 \(\{4,5,6\}\) \fi
                    \ifnum \i=4 \(\{1,2,3\}\) \fi
                }] (u\i) at (34.61+\i*360/13:2.25cm) {};
                \else \node[uStyle, fill=black, draw=black, label={-\j*25:
                    \ifnum \i=5 \(\{4,9,10\}\) \fi
                    \ifnum \i=6 \(\{1,7,8\}\) \fi
                    \ifnum \i=7 \(\{4,5,6\}\) \fi
                    \ifnum \i=8 \(\{1,2,3\}\) \fi
                    \ifnum \i=9 \(\{4,9,10\}\) \fi
                    \ifnum \i=10 \(\{1,7,8\}\) \fi
                    \ifnum \i=11 \(\{4,5,6\}\) \fi
                    \ifnum \i=12 \(\{1,2,3\}\) \fi
                }] (u\i) at (-\j*25:2.25cm) {};
                \fi
            \fi
        }
        \foreach \i/\j in {0/1,1/2,2/3,3/4}{\draw[thick,black, decorate, decoration={snake}] (u\i) -- (u\j);}
        \foreach \i/\j in {4/5,8/9,12/0}{\draw[black] (u\i) -- (u\j);}
        \foreach \i/\j in {5/6,6/7,7/8,9/10,10/11,11/12}{\draw[thick, black, decorate, decoration={zigzag, segment length=1mm, amplitude=0.5mm}] (u\i) -- (u\j);}
    \end{scope}
    
\end{tikzpicture}
\caption{Using \(3\)-tone \(10\)-colorings of \(C_4\) and \(C_5\) to show \(\tau_3(C_{13})\leq 10\).}
\label{fig: t-tone coloring cycles example}
\end{figure}
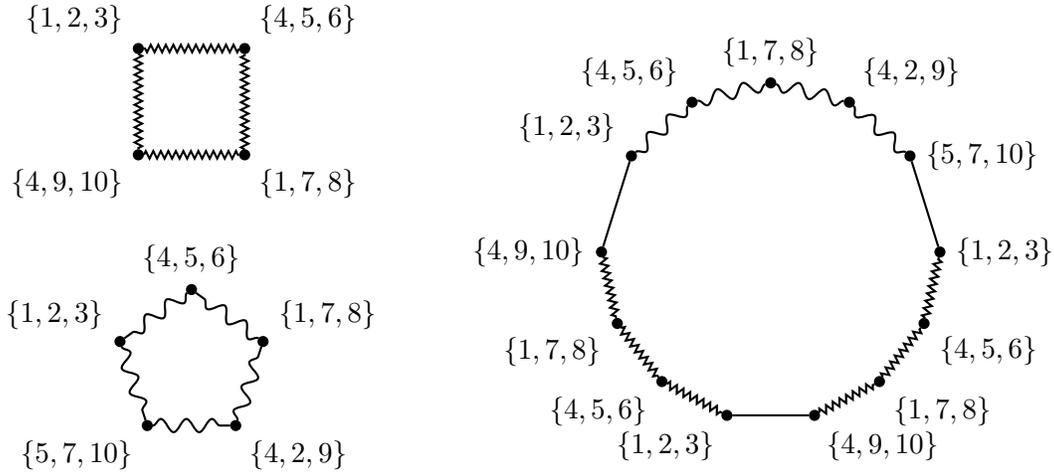
We use Lemma~\ref{lem: concatenating colored small cycles to color cycles} to
prove our next three theorems, which show that $\tau_t(C_n)=\tau_t(P_n)$ for
all $t\in \{3,4,5\}$, for all but a small (finite) number of values of $n$.

\begin{theorem}
\label{thm: 3-tone coloring cycles}
\[
\tau_3(C_n)=
\left\{\begin{tabular}{rl} 
\(10\)& if \(n\in\{4,5\}\)\\ 
\(9\)& if \(n\in\{3,7,10,13\}\)\\ 
\(8\)& \text{otherwise}\\ 
\end{tabular}\right.\]
\end{theorem}
\begin{proof}
It is easy to check that $\tau_3(P_3)=8$.  So $\tau_3(C_n)\ge \tau_3(P_3)=8$ for
all $n\ge 3$.
Pan and Wu~\cite{w} showed that \(\tau_3(C_n)=9\) when \(n\in\{3,7\}\) and
that \(\tau_3(C_n)=10\) when \(n\in\{4,5\}\). 
So we assume below that $n=6$ or $n\ge 8$.  The case $n\in \{10,13\}$ is
exceptional, so we defer it briefly to handle the general case.
In Lemma~\ref{lem: concatenating colored small cycles to color cycles},
we let \(\mathcal{C}=\{6, 8, 9, 11\}\) 
and take \(\varphi_{k}\) as described below.
\[\begin{tabular}{ll}
\(\varphi_{6}\):& \(-123-456-178-234-156-478-\)\\ 
\(\varphi_{8}\):& \(-123-456-178-234-568-127-345-678-\)\\
\(\varphi_{9}\):& \(-123-456-178-234-568-174-238-156-478-\)\\
\(\varphi_{{11}}\):& \(-123-456-178-234-568-127-634-578-126-345-678-\)\\
\end{tabular}\]

So it remains to show that $n$ can be written as an integer linear combination
of elements of $\C$ whenever $n\ge 3$ and $n\notin\{3,4,5,7,10,13\}$.
To see this, we consider the integer linear combinations, $6, 8, 9, 11, 6+6, 6+8, 6+9,
8+8, 8+9, 9+9, 8+11$ and note that every larger integer can be written as one
of the final 6, plus some multiple of 6.

Now assume $n\in \{10,13\}$.  To see that $\tau_3(C_n)\le 9$, consider the
two following 3-tone 9-colorings.
\begin{alignat*}{2}
\text{3-tone 9-coloring of $C_{10}$}: &-123-456-178-369-458-279-368-245-169-578-\\
\text{3-tone 9-coloring of $C_{13}$}: &-123-456-178-369-458-279-368-459-\\
&-278-369-245-168-579-
\end{alignat*}

Finally, we show that $\tau_3(C_n)>8$.  Assume the contrary, let $\vph$ be a
3-tone 8-coloring of $C_n$, and let $c_i$ denote the number of vertices
receiving color $i$ under $\vph$ for each $i\in[8]$.  Let
$s:=(n-1)/3$.  It is straightforward to check that, for at least $(c_i-s)2$
pairs of vertices at distance 2, both vertices receive color $i$.  Note that
$\sum_{i=1}^8c_i=3n=9s+3$.  Further, $\sum_{i=1}^8(c_i-s)2=18s+6-16s=2s+6$.
Observe that $C_n$ has precisely $n=3s+1$ pairs of vertices at distance 2. 
Since $n\in\{10,13\}$, we have $s\in\{3,4\}$, so $2s+6>3s+1$.  Thus, by
Pigeonhole some pair of vertices at distance 2 receive two common colors under
$\vph$, a contradiction.
%
\end{proof}

\begin{theorem}
\label{thm: 4-tone coloring cycles}
\[\tau_4(C_n)=\left\{\begin{tabular}{cl} 
\(15\)& if \(n=5\)\\ 
\(14\)& if \(n=4\)\\ 
\(13\)& if \(n=7\)\\ 
\(12\)& \text{otherwise}\\ 
\end{tabular}\right.\]
\end{theorem}
\begin{proof}
We have \(\tau_4(C_n)\geq \tau_4(P_n)=12\). Using results from~\cite{w}, we
have \(\tau_4(C_3)=12\), \(\tau_4(C_4)=14\), \(\tau_4(C_5)=15\), and
\(\tau_4(C_7)=13\). 
We let \(\mathcal{C}=\{6, 8, 9, {10}, {11}, {12}, {13}\}\) and take \(\varphi_{k}\) as described below.
\[\begin{tabular}{ll}
\(\varphi_{{6}}\):& \(-1,2,3,4-5,6,7,8-1,9,10,11-2,3,5,12-4,6,7,9-8,10,11,12-\)\\ 
\(\varphi_{{8}}\):& \(-1,2,3,4-5,6,7,8-1,9,10,11-2,3,5,12-4,7,8,11-1,3,6,10-2,5,8,9\)\\
&\(-7,10,11,12-\)\\
\(\varphi_{{9}}\): & \(-1,2,3,4-5,6,7,8-1,9,10,11-2,3,5,12-4,7,8,11-3,6,9,10-1,4,5,12-\)\\ & \(-2,7,8,10-6,9,11,12-\)\\
\(\varphi_{{10}}\): &\(-1,2,3,4-5,6,7,8-1,9,10,11-2,3,5,12-4,7,8,11-6,9,10,12-1,3,5,11-\)\\ &\(-2,4,8,12-3,6,7,10-5,9,11,12\)\\
\(\varphi_{{11}}\):&\(-1,2,3,4-5,6,7,8-1,9,10,11-2,3,5,12-1,4,6,7-5,8,9,10-2,3,7,11-\)\\ &\(-4,6,8,12-1,3,5,10-2,6,7,9-8,10,11,12-\)\\
\(\varphi_{{13}}\): &\(-1,2,3,4-5,6,7,8-1,9,10,11-2,3,5,12-4,7,8,11-6,9,10,12-1,3,5,11-\)\\ &\(-2,7,8,12-4,9,10,11-3,5,6,12-1,2,8,11-4,6,7,10-5,9,11,12-\)
\end{tabular}\]
So it remains to show that $n$ can be written as an integer linear combination
of elements of $\C$ whenever $n\geq 3$ and $n\notin\{3,4,5,7\}$.
To see this, we consider the integer linear combinations, $6, 8, 9, 10, 11, 6+6,
13$ and note that every larger integer can be written as one
of the final 6, plus some multiple of 6.
\end{proof}

\begin{theorem}
\label{thm: 5-tone coloring cycles}
\[\tau_5(C_n)=\left\{\begin{tabular}{cl} 
\(20\)& if \(n=5\)\\ 
\(18\)& if $n\in\{4,6\}$\\ 
\(17\)& if $n\in\{7,9\}$\\ 
\(15\)& if \(n=3\)\\ 
\(16\)& \text{otherwise}\\ 
\end{tabular}\right.\]
\end{theorem}
\begin{proof}
We have \(\tau_5(C_n)\geq \tau_5(P_n)=16\) when \(n\geq 4\). Using results from~\cite{w}, we have \(\tau_5(C_3)=15\), \(\tau_5(C_4)=18\), \(\tau_5(C_5)=20\), \(\tau_5(C_6)=18\), and \(\tau_5(C_7)=17\). We let \(\mathcal{C}=\{8, 10, 11, 12, 13, 14, 15, 17\}\) and take \(\varphi_k\) as described below.
\[\begin{tabular}{cl}
\(\varphi_8\): & \(-1,2,3,4,5-6,7,8,9,10-1,11,12,13,14-6,2,3,15,16-4,5,9,10,14-\)\\ &\(-1,3,7,8,13-2,6,10,11,12-9,13,14,15,16-\)\\
\(\varphi_{10}\): & \(-1,2,3,4,5-6,7,8,9,10-1,11,12,13,14-6,2,3,15,16-4,5,9,10,14-\)\\ &\(-7,8,12,13,16-1,5,6,11,15-2,3,9,10,16-4,7,8,11,14-6,12,13,15,16-\)\\
\(\varphi_{11}\):& \(-1,2,3,4,5-6,7,8,9,10-1,11,12,13,14-6,2,3,15,16-7,8,4,5,11-\)\\
&\(-1,6,9,10,14-7,12,13,15,16-2,3,5,8,14-1,4,7,10,11-6,2,9,12,13-\)\\
&\(-8,11,14,15,16-\)\\
\(\varphi_{12}\):& \(-1,2,3,4,5-6,7,8,9,10-1,11,12,13,14-6,2,3,15,16-1,7,8,4,5-\)\\
&\(-6,11,12,9,10-1,2,3,13,14-6,7,8,15,16-1,11,12,4,5-6,2,3,9,10-\)\\
&\(-1,7,8,13,14-6,11,12,15,16-\)\\
\(\varphi_{13}\):& \(-1,2,3,4,5-6,7,8,9,10-1,11,12,13,14-6,2,3,15,16-4,5,9,10,13-\)\\ &\(-1,7,8,11,15-2,6,10,12,14-3,4,7,13,16-5,9,10,11,15-1,2,8,12,16-\)\\ &\(-4,5,6,7,14-3,8,10,11,13-9,12,14,15,16-\)\\
\(\varphi_{14}\):& \(-1,2,3,4,5-6,7,8,9,10-1,11,12,13,14-6,2,3,15,16-4,5,9,10,13-\)\\ &\(-1,7,8,11,15-2,6,10,12,14-3,4,7,13,16-5,9,10,11,15-1,2,8,12,16-\)\\ &\(-3,5,6,13,14-1,4,7,10,15-2,8,9,11,14-6,12,13,15,16\)\\
\(\varphi_{15}\):& \(-1,2,3,4,5-6,7,8,9,10-1,11,12,13,14-6,2,3,15,16-4,5,9,10,14-\)\\ &\(-7,8,12,13,16-1,6,11,14,15-2,3,9,10,16-4,5,12,13,15-7,8,11,14,16-\)\\ &\(-1,6,9,10,15-2,3,12,13,16-4,5,8,10,14-1,7,9,11,13-6,12,14,15,16-\)\\
\(\varphi_{17}\):& \(-1,2,3,4,5-6,7,8,9,10-1,11,12,13,14-6,2,3,15,16-4,5,9,10,13-\)\\ &\(-1,7,8,11,15-2,6,10,12,14-3,4,7,13,16-5,9,10,11,15-1,2,8,12,16-\)\\ &\(-3,5,6,13,14-1,4,7,10,15-3,8,9,11,16-2,5,12,14,15-1,3,6,10,13-\)\\
&\(-4,7,9,11,14-8,12,13,15,16-\)\\
\end{tabular}\]

So it remains to show that $n$ can be written as an integer linear combination
of elements of $\C$ whenever $n\geq3$ and $n\not=9$.
To see this, we consider the integer linear combinations, $8,10,11,12,13,14,15,8+8,17,8+10, 8+11, 10+10, 10+11, 11+11, 8+15, 8+8+8, 10+15$ and note that every larger integer can be written as one of the final 8, plus some multiple of 8.\par
Now assume that $n=9$.  To see that $\tau_5(C_9)\le 17$, consider the following 5-tone 17-coloring.
\begin{alignat*}{2}
\text{5-tone 17-coloring of $C_{9}$}: &-1,2,3,4,5-6,7,8,9,10-1,11,12,13,14-6,2,3,15,16-\\
&-4,5,7,9,12-1,8,10,11,15-2,4,6,13,14-3,7,8,12,16-\\
&-9,11,13,15,17-
\end{alignat*}

Finally, we will prove that $\tau_5(C_9)\geq17$. Assume, to the contrary, that
\(C_9\) has a 5-tone 16-coloring. Note that each color appears on at most 4
vertices. Each color must appear on at least one vertex, since $\tau_5(C_9)\ge
\tau_5(P_4)=16$.
For each \(i\in[4]\), let \(s_i\) denote the
number of colors used on exactly \(i\) vertices. So we have
\(\sum_{i=1}^4s_i=16\) and \(\sum_{i=1}^4 is_i=9(5)=45\). Further, let \(s_3'\)
denote the number of colors used on exactly 3 vertices, where some pair is at
distance 2, and let \(s_3''\) denote the number of colors used on exactly 3
vertices, where each pair is distance 3. Note that each color used on 4
vertices is used on 3 pairs of vertices at distance 2. Since \(C_9\) has 9
pairs of vertices at distance 2, and each pair can share at most 1 common
color, we get \(3s_4+s_3'\leq 9\). Similarly, by considering vertex pairs with
a common color that are at distance 3, we get \(s_3'+3s_3''\leq 18\).
Multiplying the first inequality by 2, adding it to the second inequality, and
dividing by 3 (recalling \(s_3'+s_3''=s_3\)) gives
\[2s_4+s_3\leq 12. \tag{\(\ast\)}\]
Recall that \(\sum_{i=1}^4s_i=16\) and \(\sum_{i=1}^4is_i=9(5)=45\).
Multiplying the first equation by 3 and subtracting the second gives
\(2s_1+s_2-s_4=3\). Adding this to \((\ast)\) gives \(2s_1+s_2+s_3+s_4\leq
12+3=15\). Since \(s_1\geq 0\), this contradicts the first equation, and this
contradiction finishes the proof. 
\end{proof}

We conclude this section with a bold conjecture.

\begin{conj}
\label{conj: strong cycle conjecture}
For each \(t\geq2\) there exists \(N\in\NN\) such that \(\tau_t(C_n)=\tau_t(P_n)\) for all \(n\geq N\).
\end{conj}

\section{3-Tone, 4-Tone, and 5-Tone Coloring of Grid Graphs}
\label{grids-sec}


In this section we will consider the \(t\)-tone chromatic number of grid graphs for each \(t\in\{3,4,5\}\).

Cooper and Wash~\cite[Theorem 5]{cw} showed that \(\tau_2(P_n\square P_m)=6\)
for all \(n,m\geq 2\). It is useful in their proof, and in the following three
theorems, to imagine the grid graph as being drawn in the first quadrant of the
\(xy\)-plane with vertices as integer points. Now their proof 
can be viewed as coloring lines of slope 1 by cycling through the
colors 1, 2, 3 and coloring lines of slope $-1$ by cycling through the colors 4, 5, 6.

For Theorem~\ref{thm: 3-tone coloring grid graphs}, the proof can be
viewed as coloring the lines of slope 1 and slope $-1$ as above, but
also coloring lines of slope 2. This theorem improves a result in~\cite[Theorem 8]{cw}.
For Theorem~\ref{thm: 4-tone coloring grid graphs}, the proof can be viewed
as coloring the lines of slope 1, slope $-1$, and slope 2 as in
Theorem~\ref{thm: 3-tone coloring grid graphs}, but further coloring lines of slope \(-\frac{1}{2}\).
Finally, for Theorem~\ref{thm: 5-tone coloring grid graphs}, the proof can also
be viewed as coloring the lines of slope 1, slope $-1$, slope 2, and slope
\(-\frac{1}{2}\) as in Theorem~\ref{thm: 4-tone coloring grid graphs}, but adding colors to lines of slope 1.

For the following three theorems we consider the vertices of
\(P_m\square P_n\) as integer points on the \(xy\)-plane where a vertex
\((x_i,y_j)\) is denoted by \((i,j)\) with \(1\leq i\leq m\) and \(1\leq
j\leq n\). For all vertices \((i_1,j_1)\) and \((i_2,j_2)\) in
\(V(P_m\square P_n)\), note that the distance between them is exactly
\(\card{i_1-i_2}+\card{j_1-j_2}\). 
\begin{theorem}
\label{thm: 3-tone coloring grid graphs}
\(\tau_3(P_m\square P_n)=10\) for all integers \(m\) and \(n\) with \(2\leq m\leq n\).
\end{theorem}
\begin{proof}
Lemmas~\ref{prop: t-tone lower bound from subgraphs} and~\ref{prop: t-tone coloring number of C_4} imply that \(10=\tau_3(C_4)\leq\tau_3(P_m\square P_n)\). So it suffices to construct a 3-tone 10-coloring of \(P_m\square P_n\). Let \(f: V(P_m\square P_n)\rightarrow \binom{[10]}{3}\) where we write \(f((i,j))\) as \(f(i,j)\) and we let \(f(i,j):=\{f_1(i,j),f_2(i,j),f_3(i,j)\}\),  where
\begin{alignat*}{2}
f_1(i,j)&:=(i-j)\text{ mod } 3 \\
f_2(i,j)&:=((i+j)\text{ mod } 3) +3 \tag{\(\ast\)}\\
f_3(i,j)&:=((2i+j)\text{ mod } 4) +6.
\end{alignat*}

Denote \(v\) by \((i_1,j_1)\) and \(w\) by \((i_2,j_2)\). 
It suffices to prove the 
following three claims.

\noindent \underline{Claim 1}: If \(\card{f(v)\cap f(w)}=3\), then \(d(v,w)\geq 4\).

If \(\card{f(v)\cap f(w)}=3\), then \(f_i(v)=f_i(w)\) for all \(i\in[3]\). So \((i_1-j_1)\equiv (i_2-j_2)\mod 3\) and \((i_1+j_1)\equiv (i_2+j_2)\mod 3\). Thus \(i_1\equiv i_2\mod 3\) and \(j_1\equiv j_2\mod 3\). If \(d(v,w)\leq 3\) and \(v\not= w\), then \(i_1\equiv i_2 \pm 3\) and \(j_1=j_2\) or else \(i_1=i_2\) and \(j_1=j_2\pm 3\). But now \((2i_1+j_1)\not\equiv(2i_2+j_2)\mod 4\).

\noindent \underline{Claim 2}: If \(\card{f(v)\cap f(w)}=2\), then \(d(v,w)\geq 3\).

Assume \(\card{f(v)\cap f(w)}=2\). If \(\{f_1(v), f_2(v)\}=\{f_1(w),
f_2(w)\}\), then the argument in Claim 1 still holds. Instead we assume
\(f_3(v)=f_3(w)\) and \(d(v,w)\leq 2\). Thus \(i_1=i_2\pm 2\) and \(j_1=j_2\),
but now \(f_1(v)\not=f_2(v)\) and \(f_2(v)\not=f_2(w)\), a contradiction.

\noindent \underline{Claim 3}: If \(\card{f(v)\cap f(w)}=1\), then \(d(v,w)\geq 2\).

Assume that \(d(v,w)=1\). So either \(i_1=i_2\) and \(j_1-j_2=\pm 1\) or else \(j_1=j_2\) and \(i_1-i_2=\pm 1\). Now clearly \(f_i(v)\not =f_i(w)\) for all \(i\in[3]\), a contradiction.
\end{proof}

\begin{theorem}
\label{thm: 4-tone coloring grid graphs}
\(\tau_4(P_m\square P_n)=14\) for integers \(m\) and \(n\) with \(2\leq m\leq n\).
\end{theorem}
\begin{proof}
Lemmas~\ref{prop: t-tone lower bound from subgraphs} and \ref{prop: t-tone coloring number of C_4} imply that 
\(14=\tau_4(C_4)\leq\tau_4(P_m\square P_n)\). So it suffices to construct a
4-tone 14-coloring of \(P_m\square P_n\). Let \(f: V(P_m\square P_n)\rightarrow
\binom{[14]}{4}\), where we write \(f((i,j))\) as \(f(i,j)\) and we let
\(f(i,j):=\{f_1(i,j),f_2(i,j),f_3(i,j),f_4(i,j)\}\), where
\begin{alignat*}{2}
f_1(i,j)&:=(i-j)\text{ mod } 3 \\
f_2(i,j)&:=((i+j)\text{ mod } 3) + 3 \\
f_3(i,j)&:=((2i+j)\text{ mod } 4) + 6 \tag{\(\ast\ast\)}\\
f_4(i,j)&:=((i+2j)\text{ mod } 4) + 10.
\end{alignat*}

Denote \(v\) by \((i_1,j_1)\) and \(w\) by \((i_2,j_2)\). Assume \(d(v,w)=1\).
It suffices to prove the following four claims.

\noindent\underline{Claim 1}: If \(\card{f(v)\cap f(w)}=4\), then \(d(v,w)\geq 5\). 

Assume \(\card{f(v)\cap f(w)}=4\). So \(f_i(v)=f_i(w)\) for all \(i\in[4]\).
Claim 1 in Theorem~\ref{thm: 3-tone coloring grid graphs} implies \(d(v,w)\geq
4\). Suppose \(d(v,w)=4\). Since \(f_4(v)=f_4(w)\) we have \(i_1-i_2\equiv
0\mod 4\) and \(j_1=j_2\), or \(j_1-j_2\equiv 0\mod 4\) and \(i_1=i_2\). In
either case this implies \(f_k(v)\not=f_k(w)\) for each \(k\in\{1,2\}\), a
contradiction.

\noindent\underline{Claim 2}: If \(\card{f(v)\cap f(w)}=3\), then \(d(v,w)\geq 4\). 

Assume \(\card{f(v)\cap f(w)}=3\). Claim 1 in Theorem~\ref{thm: 3-tone coloring
grid graphs} implies \(f_4(v)=f_4(w)\); and Claim 2 in Theorem~\ref{thm: 3-tone
coloring grid graphs} implies \(d(v,w)\geq 3\). Suppose \(d(v,w)=3\). If
\(f_3(v)\not=f_3(w)\), then \(i_1\equiv i_2\mod 3\) and \(j_1\equiv j_2\mod
3\), but then \(f_4(v)\not=f_4(w)\), a contradiction. 
If \(f_3(v)=f_3(w)\), then \(i_1-i_2\equiv
j_1-j_2 \mod4\), which implies \(f_1(v)\not=f_1(w)\) and \(f_2(v)\not=f_2(w)\),
contradicting \(\card{f(v)\cap f(w)}=3\).

\noindent\underline{Claim 3}: If \(\card{f(v)\cap f(w)}=2\), then \(d(v,w)\geq 3\).

Assume \(\card{f(v)\cap f(w)}=2\). 
If $f_4(v)\ne f_4(w)$,
then by Claim 2 in Theorem~\ref{thm: 3-tone coloring
grid graphs} we know \(d(v,w)\geq 3\). So we may assume \(f_4(v)=f_4(w)\) and
\(f_k(v)=f_k(w)\) for some single $k\in[3]$. From Claim 3 in
Theorem~\ref{thm: 3-tone coloring grid graphs} we have that \(d(v,w)\geq 2\).
Suppose \(d(v,w)=2\). Since \(f_4(v)=f_4(w)\) it must be that \(i_1=i_2\). So
\(j_1-j_2\equiv 2\mod 4\); but now \(f_k(v)\not=f_k(w)\) for
all \(k\in\{1,2\}\), a contradiction.

\noindent\underline{Claim 4}: If \(\card{f(v)\cap f(w)}=1\), then \(d(v,w)\geq 2\).

Assume \(\card{f(v)\cap f(w)}=1\). If \(f_4(v)\not=f_4(w)\), then Claim 3 in
Theorem~\ref{thm: 3-tone coloring grid graphs} implies \(d(v,w)\geq 2\). So
\(f_4(v)=f_4(w)\), which implies \(d(v,w)\geq 2\).
\end{proof}
 
\begin{theorem}
\label{thm: 5-tone coloring grid graphs}
\(20\leq\tau_5(P_m\square P_n)\leq 22\) for all \(2\leq m< n\).
\end{theorem}
\begin{proof}
Using Lemma~\ref{prop: t-tone coloring number of C_4}  when \(t\geq 5\) implies
\(\tau_t(P_2\square P_3)=6t-10\); in fact, an optimal \(t\)-tone coloring
\(\varphi\) of \(P_2\square P_3\) is unique up to relabelling. This fact
combined with Lemma~\ref{prop: t-tone lower bound from subgraphs} implies \(20=\tau_5(P_2\square P_3)\leq\tau_5(P_m\square P_n)\).\par
It now suffices to show a 5-tone 22-coloring of \(P_m\square P_n\). Let \(f: V(P_m\square P_n)\rightarrow \binom{[22]}{5}\) where we will denote \(f((i,j))\) as \(f(i,j)\) and define \(f(i,j):=\{f_1(i,j),f_2(i,j),f_3(i,j),f_4(i,j), f_5(i,j)\}\) where 
\begin{alignat*}{2}
f_1(i,j)&:=(i-j)\text{ mod } 3 \\
f_2(i,j)&:=((i+j)\text{ mod } 3) + 3 \\
f_3(i,j)&:=((2i+j)\text{ mod } 4) + 6\\
f_4(i,j)&:=((i+2j)\text{ mod } 4) + 10\\
f_5(i,j)&:=((i+3j) \text{ mod } 8) + 14.
\end{alignat*}
Let \(v=(i_1,j_1)\), \(w=(i_2,j_2)\), and \(q=\card{f(v)\cap f(w)}\). If
\(q\in\{0,\dots,4\}\) and \(f_5(v)\not=f_5(v)\), then \((\ast\ast)\) and the
claims in Theorem~\ref{thm: 4-tone coloring grid graphs} imply \(d(v,w)\geq
q+1\). So we assume \(f_5(v)=f_5(v)\). This implies \(d(v,w)\geq 4\) since
otherwise \(((i_1-i_2)+3(j_1-j_2))\text{ mod }8\not = 0\). So it suffices to
prove the following three claims.

\noindent\underline{Claim 1}: If \(\card{f(v)\cap f(w)}=4\), then \(d(v,w)\geq 5\). 

Assume \(\card{f(v)\cap f(w)}=4\). Suppose \(d(v,w)=4\). Since
\(f_5(v)=f_5(w)\), either: \(i_1-i_2=\pm1\) and \(j_1-j_2=\mp 3\); or
\(i_1-i_2=\pm2\) and \(j_1-j_2=\pm2\); or  \(i_1-i_2=\pm3\) and
\(j_1-j_2=\mp1\). In all cases \(f_2(v)\not=f_2(w)\) and \(f_3(v)\not=f_3(w)\),
a contradiction to \(\card{f(v)\cap f(w)}=4\).

\noindent\underline{Claim 2}: If \(\card{f(v)\cap f(w)}=5\), then \(d(v,w)\geq 6\). 

Assume \(\card{f(v)\cap f(w)}=5\). Claim 1 implies \(d(v,w)\geq 5\). 
So $|i_1-i_2|+|j_1-j_2|=5$.  But now $f_5(v)\ne f_5(w)$, a contradiction.
\end{proof}


\bibliographystyle{amsplain} 
\end{document}